\documentclass[12pt,oneside,reqno]{amsart}

\usepackage{amsmath,amssymb,amsthm,textcomp}
\usepackage{amsfonts,graphicx}
\usepackage[mathscr]{eucal}
\usepackage{color}
\usepackage{csquotes}
\usepackage[backend=bibtex,%
firstinits=true,%
doi=false,%
isbn=true,%
url=false,%
maxnames=99]{biblatex}%

\AtEveryBibitem{\clearfield{issn}}
\AtEveryCitekey{\clearfield{issn}}
\numberwithin{equation}{section}
\DeclareNameAlias{sortname}{last-first}
\theoremstyle{definition}
\usepackage{mathtools}
\numberwithin{equation}{section}

\newcommand{\ncom}{\newcommand}

\ncom{\beq}{\begin{equation}}
\ncom{\eeq}{\end{equation}}
\ncom{\bea}{\begin{eqnarray*}}
\ncom{\eea}{\end{eqnarray*}}
\ncom{\beqa}{\begin{eqnarray}}
\ncom{\eeqa}{\end{eqnarray}}
\ncom{\nno}{\nonumber}
\ncom{\non}{\nonumber}
\ncom{\ds}{\displaystyle}
\ncom{\half}{\frac{1}{2}}
\ncom{\mbx}{\makebox{.25cm}}
\ncom{\hs}{\mbox{\hspace{.25cm}}}
\ncom{\rar}{\rightarrow}
\ncom{\Rar}{\Rightarrow}
\ncom{\noin}{\noindent}
\ncom{\bc}{\begin{center}}
\ncom{\ec}{\end{center}}
\ncom{\sz}{\scriptsize}
\ncom{\rf}{\ref}
\ncom{\s}{\sqrt{2}}
\ncom{\sgm}{\sigma}
\ncom{\Sgm}{\Sigma}
\ncom{\psgm}{\sigma^{\prime}}
\ncom{\dt}{\delta}
\ncom{\Dt}{\Delta}
\ncom{\lmd}{\lambda}
\ncom{\Lmd}{\Lambda}
\ncom{\Th}{\Theta}
\ncom{\e}{\eta}
\ncom{\eps}{\epsilon}
\ncom{\pcc}{\stackrel{P}{>}}
\ncom{\lp}{\stackrel{L_{p}}{>}}
\ncom{\dist}{{\rm\,dist}}
\ncom{\sspan}{{\rm\,span}}
\ncom{\re}{{\rm Re\,}}
\ncom{\im}{{\rm Im\,}}
\ncom{\sgn}{{\rm sgn\,}}
\ncom{\ba}{\begin{array}}
\ncom{\ea}{\end{array}}
\ncom{\hone}{\mbox{\hspace{1em}}}
\ncom{\htwo}{\mbox{\hspace{2em}}}
\ncom{\hthree}{\mbox{\hspace{3em}}}
\ncom{\hfour}{\mbox{\hspace{4em}}}
\ncom{\vone}{\vskip 2ex}
\ncom{\vtwo}{\vskip 4ex}
\ncom{\vonee}{\vskip 1.5ex}
\ncom{\vthree}{\vskip 6ex}
\ncom{\vfour}{\vspace*{8ex}}
\ncom{\norm}{\|\;\;\|}
\ncom{\integ}[4]{\int_{#1}^{#2}\,{#3}\,d{#4}}
\ncom{\vspan}[1]{{{\rm\,span}\{ #1 \}}}
\ncom{\dm}[1]{ {\displaystyle{#1} } }
\ncom{\ri}[1]{{#1} \index{#1}}

\newtheorem{theorem}{\bf Theorem}[section]
\newtheorem{remark}{\bf Remark}[section]
\newtheorem{proposition}{Proposition}[section]
\newtheorem{lemma}{Lemma}[section]

\newtheorem{definition}{Definition}[section]
\newtheoremstyle
    {remarkstyle}
    {}
    {11pt}
    {}
    {}
    {\bfseries}
    {:}
    {     }
    {\thmname{#1} \thmnumber{#2} }

\theoremstyle{remarkstyle}

\def\eps{\varepsilon}

\def\E{{\mathbb E}}

\begin{document}
\title{Fractional Skellam process of order $k$}
\author[Kuldeep Kumar Kataria]{Kuldeep Kumar Kataria}
\address{Kuldeep Kumar Kataria, Department of Mathematics, Indian Institute of Technology Bhilai, Raipur 492015, India.}
 \email{kuldeepk@iitbhilai.ac.in}
\author[Mostafizar Khandakar]{Mostafizar Khandakar}
\address{Mostafizar Khandakar, Department of Mathematics, Indian Institute of Technology Bhilai, Raipur 492015, India.}
\email{mostafizark@iitbhilai.ac.in}
\subjclass[2010]{Primary: 60G22; Secondary: 60G55}
\keywords{Skellam distribution; Poisson process of order $k$; L\'evy subordinator; LRD property.}
\date{March 16, 2021}
\begin{abstract}
We introduce and study a fractional version of the Skellam process of order $k$ by time-changing it with  an independent inverse stable subordinator. We call it the fractional Skellam process of order $k$ (FSPoK). An integral representation for its one-dimensional distributions and their governing system of fractional differential equations are obtained.  We derive the probability generating function, mean, variance and covariance of the FSPoK which are utilized to establish its long-range dependence property. Later, we considered two time-changed versions of the FSPoK. These are obtained by time-changing the FSPoK by an independent L\'evy subordinator and its inverse. Some distributional properties and particular cases are discussed for these time-changed processes. 
\end{abstract}

\maketitle
\section{Introduction}
Let $\{N(t)\}_{t\ge0}$ be a Poisson process with intensity $k\lambda$ and   $\{X_{i}\}_{i\ge1}$ be a sequence of independent and identically distributed (iid) discrete uniform random variables with support $S=\{1,2,\dots,k\}$.  Consider the following compound Poisson process: 
\begin{equation*}
N^{k}(t)=\sum_{i=1}^{N(t)}X_{i},
\end{equation*}
where $\{N(t)\}_{t\ge0}$ is independent of $\{X_{i}\}_{i\ge1}$. The process $\{N^{k}(t)\}_{t\ge0}$ is known as the Poisson process of order $k$ (PPoK) (see Kostadinova and Minkova (2013)). For $k=1$, the PPoK reduces to the Poisson process $\{N(t)\}_{t\ge0}$ with intensity $\lambda$. In ruin theory, the PPoK is used to model the number of claims where these arrive in groups of size $k$.  

Recently, Gupta {\it et al.} (2020) introduced and studied a L\'evy process $\{S^{k}(t)\}_{t\ge0}$ by considering the difference of two  PPoK, that is, 
\begin{equation}\label{skellam}
S^{k}(t)=N_{1}^{k}(t)-N_{2}^{k}(t),
\end{equation}
where $\{N_{1}^{k}(t)\}_{t\ge0}$ and  $\{N_{2}^{k}(t)\}_{t\ge0}$ are independent PPoK with intensities $\lambda_{1}>0$ and $\lambda_{2}>0$, respectively. It is called  the
Skellam process of order $k$ (SPoK) whose state probabilities $p^{k}(n,t)=\mathrm{Pr}\{S^{k}(t)=n\}$, $n\in\mathbb{Z}$  satisfy the following system of differential equations:
\begin{equation}\label{skellam gov}
\frac{\mathrm{d}}{\mathrm{d}t}p^{k}(n,t)=-k(\lambda_{1}+\lambda_{2})p^{k}(n,t)+\lambda_{1}\sum_{j=1}^{k}p^{k}(n-j,t)+\lambda_{2}\sum_{j=1}^{k}p^{k}(n+j,t),
\end{equation}
with initial conditions $p^{k}(0,0)=1$ and $p^{k}(n,0)=0, \ n\neq0$. For $k=1$, the SPoK reduces to an integer-valued L\'evy process, namely, the Skellam process $\{S(t)\}_{t\ge0}$ (see Barndorff-Nielsen {\it et al.} (2012)). For each $t\ge0$, the random variable $S(t)$ has the Skellam distribution (see  Skellam (1946)). It has several real life applications, for example,  it is used as a sensor noise model for cameras (see Hwang {\it et al.} (2007)) and in modeling the score differences  in a game of two competing teams (see Karlis and Ntzoufras (2009)).

In the past two decades, the time-changed point processes attracted the interest of several researchers due to their potential applications in different fields such as finance, hydrology, econometrics, {\it etc.} The Poisson process time-changed by a stable subordinator and by the inverse of a  stable subordinator leads to the space fractional Poisson process (see Orsingher and Polito (2012)) and the time fractional Poisson process (see Meerschaert {\it et al.} (2011)), respectively. For other time-changed version of the Poisson process, we refer the reader to Beghin (2012), Orsingher and Toaldo (2015), Aletti {\it et al.} (2018), {\it etc.}, and the references therein. 

 A subordinator $\{D_f(t)\}_{t\ge0}$ is a one-dimensional L\'evy process  whose Laplace transform is given by (see Applebaum (2009), Section 1.3.2)
\begin{equation*}
\mathbb{E}\left(e^{-sD_f(t)}\right)=e^{-tf(s)},
\end{equation*}
where the function 
\begin{equation}\label{fs}
f(s)=bs+\int_0^\infty \left(1-e^{-sx}\right)\,\mu(\mathrm{d}x),\ \ s>0,
\end{equation}
is called the Bern\v stein function. Here, $b\ge0$ is the drift coefficient and $\mu$ is a non-negative  L\'evy measure that satisfies $\mu([0,\infty))=\infty$ and
$\displaystyle\int_0^\infty (x\wedge1)\,\mu(\mathrm{d}x)<\infty$ where $s \wedge t=\min\{s,t\}$. The sample paths of a subordinator $\{D_f(t)\}_{t\ge0}$ are non-decreasing and $D_f(0)=0$ almost surely ({\it a.s.}). The first hitting time of $\{D_f (t)\}_{t\ge0}$ is called the inverse subordinator. It is defined as
\begin{equation*}
H_f (t)\coloneqq\inf\{r\ge0: D_f (r)> t\}, \ \ t\ge0.
\end{equation*}
A driftless subordinator, {\it i.e.}, $b=0$ with $f(s)=s^{\alpha}$, $0<\alpha<1$ is known as the stable subordinator. Its first hitting time $\{Y_{\alpha}(t)\}_{t\ge0}$ is called the inverse stable subordinator.

Recently, Gupta {\it et al.} (2020) introduced a time-changed SPoK by time-changing it with an independent subordinator. They obtained its probability mass function (pmf), mean, variance and covariance. Here, we introduce and study a time-changed version of the SPoK by time-changing it with an independent inverse stable subordinator. It is defined as 
\begin{equation*}
\mathcal{S}_{\alpha}^{k}(t)=\begin{cases}
S^{k}\left(Y_{\alpha}(t)\right), \ \ 0<\alpha<1,\\
S^{k}(t), \ \  \alpha=1.
\end{cases}
\end{equation*}
We call the process $\{\mathcal{S}_{\alpha}^{k}(t)\}_{t\ge0}$ as the fractional Skellam process of order $k$ (FSPoK). For $k=1$, the FSPoK reduces to the fractional Skellam process (FSP), for more details on FSP and its application to high frequency financial data set we refer the reader to Kerss {\it et al.} (2014). For $\alpha=k=1$, the FSPoK reduces to the classical Skellam process.

We obtain an integral representation for the one-dimensional distributions of FSPoK. Also, the governing system of fractional differential equations for its state probabilities is obtained. The probability generating function (pgf), mean, variance and covariance of the FSPoK are derived and its long-range dependence (LRD) property is established. Later, we explore its time-changed  versions.  We consider  the FSPoK time-changed by an independent L\'evy subordinator  and its inverse. We compute the mean, variance and covariance for these time-changed FSPoK. For the first time-changed version we establish the law of iterated logarithm and the LRD property under suitable restrictions on the L\'evy subordinator. Some particular cases of these time-changed versions are also considered by taking specific subordinators such as the gamma subordinator, tempered stable subordinator and inverse Gaussian subordinator. Also, we obtain the associated system of governing differential equations  for these particular cases. The results obtained in this paper generalize and complement the results of Kerss {\it et al.} (2014) and Gupta {\it et al.} (2020).

\section{Fractional Skellam process of order $k$}
In this section, we introduce and study a stochastic process, namely, the fractional Skellam process of order $k$ (FSPoK) that is obtained by time-changing the SPoK by an independent inverse stable subordinator. Let  $\{Y_{\alpha}(t)\}_{t\ge0},\ 0<\alpha<1$, be an inverse stable subordinator which is independent of SPoK $\{S^{k}(t)\}_{t\ge0}$. The FSPoK $\{\mathcal{S}_{\alpha}^{k}(t)\}_{t\ge0}$ is defined as 
\begin{equation}\label{FSPoK}
\mathcal{S}_{\alpha}^{k}(t)=\begin{cases}
S^{k}\left(Y_{\alpha}(t)\right), \ \ 0<\alpha<1,\\
S^{k}(t), \ \  \alpha=1.
\end{cases}
\end{equation}
For $k=1$, the process defined in (\ref{FSPoK}) reduces to the fractional Skellam process (FSP) which is  introduced and studied by Kerss {\it et al.} (2014). Moreover, for $\alpha=k=1$ the FSPoK reduces to the classical Skellam process.
	
In the following result we derive the system of governing differential equations that is  satisfied by the state probabilities $p_{\alpha}^{k}(n,t)=\mathrm{Pr}\{\mathcal{S}^{k}_{\alpha}(t)=n\}$ of FSPoK.
\begin{proposition}
The state probabilities $p_{\alpha}^{k}(n,t),\ n\in \mathbb{Z}$ of  FSPoK solves the following system of fractional differential equations: 
\begin{equation}\label{pmf}
\partial_{t}^{\alpha}p^{k}_{\alpha}(n,t)=-k(\lambda_{1}+\lambda_{2})p^{k}_{\alpha}(n,t)+\lambda_{1}\sum_{j=1}^{k}p^{k}_{\alpha}(n-j,t)+\lambda_{2}\sum_{j=1}^{k}p^{k}_{\alpha}(n+j,t),
\end{equation}
with the initial conditions $p^{k}_{\alpha}(0,0)=1$ and $p^{k}_{\alpha}(n,0)=0, \ n\neq0$. Here, $\partial_{t}^{\alpha}$ denotes the Caputo fractional derivative defined as (see Kilbas {\it et al.} (2006))
	\begin{equation*}
	\partial_{t}^{\alpha}f(t):=\left\{
	\begin{array}{ll}
	\dfrac{1}{\Gamma{(1-\alpha)}}\displaystyle\int^t_{0} (t-s)^{-\alpha}f'(s)\,\mathrm{d}s,\ \ 0<\alpha<1,\\\\
	f'(t),\ \ \alpha=1.
	\end{array}
	\right.
	\end{equation*}
\end{proposition}
\begin{proof}
From (\ref{FSPoK}), we have
\begin{equation}\label{ci}
p^{k}_{\alpha}(n,t)=\int_{0}^{\infty}p^{k}(n,u)h_{\alpha}(u,t)\,\mathrm{d}u,
\end{equation}	
where $h_{\alpha}(u,t)$ is the pdf of $\{Y_{\alpha}(t)\}_{t\ge0}$. For $\gamma\geq 0$, the Riemann-Liouville (R-L) derivative is defined as (see Kilbas {\it et al.} (2006))
\begin{equation}\label{RL}
D_t^{\gamma}f(t):=\left\{
\begin{array}{ll}
\dfrac{1}{\Gamma{(m-\gamma)}}\displaystyle \frac{\mathrm{d}^{m}}{\mathrm{d}t^{m}}\int^t_{0} \frac{f(s)}{(t-s)^{\gamma+1-m}}\,\mathrm{d}s,\ \ m-1<\gamma<m,\\\\
\displaystyle\frac{\mathrm{d}^{m}}{\mathrm{d}t^{m}}f(t),\ \ \gamma=m,
\end{array}
\right.
\end{equation}
where $m$ is a positive integer. Using the following result (see Meerschaert and Straka  (2013)) 
\begin{equation*}
D_{t}^{\alpha}h_{\alpha}(u,t)=-\frac{\partial}{\partial u}h_{\alpha}(u,t),
\end{equation*}
we get
\begin{align}\label{Rl}
D_{t}^{\alpha}p^{k}_{\alpha}(n,t)&=-\int_{0}^{\infty}p^{k}(n,u)\frac{\partial}{\partial u}h_{\alpha}(u,t)\mathrm{d}u\nonumber\\
&=p^{k}(n,0)h_{\alpha}(0+,t)+\int_{0}^{\infty}h_{\alpha}(u,t)\frac{\mathrm{d}}{\mathrm{d} u}p^{k}(n,u)\,\mathrm{d}u \nonumber\\
&=p^{k}(n,0)\frac{t^{-\alpha}}{\Gamma(1-\alpha)}+\int_{0}^{\infty}h_{\alpha}(u,t)\frac{\mathrm{d}}{\mathrm{d} u}p^{k}(n,u)\,\mathrm{d}u,
\end{align}
where in the last step we used $h_{\alpha}(0+,t)=t^{-\alpha}/\Gamma(1-\alpha)$ (see Meerschaert and Straka  (2013)).
The following relation holds:
\begin{equation}\label{relation}
\partial_{t}^{\alpha}p^{k}_{\alpha}(n,t)=D_{t}^{\alpha}p^{k}_{\alpha}(n,t)-p^{k}_{\alpha}(n,0)\frac{t^{-\alpha}}{\Gamma(1-\alpha)}.
\end{equation}
From (\ref{ci}), we have $p^{k}_{\alpha}(n,0)=p^{k}(n,0)$ as $h_{\alpha}(u,0)=\delta_0(u)$. Substituting (\ref{skellam gov}) in (\ref{Rl}) and then using it in (\ref{relation}), we get
\begin{align*}
\partial_{t}^{\alpha}p^{k}_{\alpha}(n,t)&=\int_{0}^{\infty}h_{\alpha}(u,t)\frac{\mathrm{d}}{\mathrm{d} u}p^{k}(n,u)\,\mathrm{d}u\\
&=\int_{0}^{\infty}\left(-k(\lambda_{1}+\lambda_{2})p^{k}(n,u)+\lambda_{1}\sum_{j=1}^{k}p^{k}(n-j,u)+\lambda_{2}\sum_{j=1}^{k}p^{k}(n+j,u)\right)h_{\alpha}(u,t)\,\mathrm{d}u\\
&=-k(\lambda_{1}+\lambda_{2})p^{k}_{\alpha}(n,t)+\lambda_{1}\sum_{j=1}^{k}p^{k}_{\alpha}(n-j,t)+\lambda_{2}\sum_{j=1}^{k}p^{k}_{\alpha}(n+j,t).
\end{align*}
This completes the proof.
\end{proof}	
On substituting $k=1$ in (\ref{pmf}), we get the system of  governing differential equations for the state probabilities  $p_{\alpha}(n,t)=\mathrm{Pr}\{\mathcal{S}_{\alpha}(t)=n\},\ n\in \mathbb{Z}$ of FSP  (see Kerss {\it et al.} (2014), Eq. (3.5)) as follows:
\begin{equation*}
\partial_{t}^{\alpha}p_{\alpha}(n,t)=-(\lambda_{1}+\lambda_{2})p_{\alpha}(n,t)+\lambda_{1}p_{\alpha}(n-1,t)+\lambda_{2}p_{\alpha}(n+1,t),
\end{equation*}
with $p_{\alpha}(0,0)=1$ and $p_{\alpha}(n,0)=0, \ n\neq0$.

To state the next result we need two special functions, namely, the Wright function $M_{\alpha}(\cdot)$, $0<\alpha<1$ and the modified Bessel function of first kind $I_{n}(\cdot)$. These are defined as follows (see Mainardi (2010), Abramowitz and  Stegun (1972)):
\begin{equation*}
M_{\alpha}(z)=\sum_{m=0}^{\infty}\frac{(-z)^{m}}{m!\Gamma(1-m\alpha-\alpha)}
\end{equation*}
and
\begin{equation*}
I_{n}(z)=\sum_{m=0}^{\infty}\frac{(z/2)^{2m+n}}{m!(m+n)!}.
\end{equation*}
\begin{theorem}
For $n\in\mathbb{Z}$, the state probability  $p_{\alpha}^{k}(n,t)$ of FSPoK is given by 
\begin{equation*}
p_{\alpha}^{k}(n,t)=
\frac{1}{t^{\alpha}}\left(\frac{\lambda_{1}}{\lambda_{2}}\right)^{n/2}\int_{0}^{\infty}e^{-ku(\lambda_{1}+\lambda_{2})}I_{|n|}\left(2uk\sqrt{\lambda_{1}\lambda_{2}}\right)M_{\alpha}\left(\frac{u}{t^{\alpha}}\right)\mathrm{d}u.
\end{equation*}
\end{theorem}	
\begin{proof}
For $n\in\mathbb{Z}$, the state probability $p^{k}(n,t)$ of SPoK is given by (see Gupta {\it et al.} (2020), Eq. (38)):
\begin{equation}\label{pknt}
p^{k}(n,t)=e^{-kt(\lambda_{1}+\lambda_{2})}\left(\frac{\lambda_{1}}{\lambda_{2}}\right)^{n/2}I_{|n|}\left(2tk\sqrt{\lambda_{1}\lambda_{2}}\right).
\end{equation}
The following result holds (see Meerschaert {\it et al.} (2015), Section 3):
\begin{equation}\label{hat}
h_{\alpha}(u,t)=\frac{1}{t^{\alpha}}M_{\alpha}\left(\frac{u}{t^{\alpha}}\right).
\end{equation}
The result follows on substituting (\ref{pknt}) and (\ref{hat})	in (\ref{ci}).
\end{proof}
The three-parameter Mittag-Leffler function is defined as (see Kilbas {\it et al.} (2006), p. 45)
\begin{equation}\label{mit}
E_{\alpha, \beta}^{\gamma}(x)\coloneqq\frac{1}{\Gamma(\gamma)}\sum_{k=0}^{\infty} \frac{\Gamma(\gamma+k)x^{k}}{k!\Gamma(k\alpha+\beta)},\ \ x\in\mathbb{R},
\end{equation}
where $\alpha>0$, $\beta>0$ and $\gamma>0$. For $\gamma=1$, it reduces to two-parameter Mittag-Leffler function. For  $\gamma=\beta=1$, it further reduces to the Mittag-Leffler function.

In the next result we obtain the pgf of FSPoK.
\begin{proposition}\label{pgf}
The pgf $G^{k}_{\alpha}(\theta,t)=\mathbb{E}\left(\theta^{\mathcal{S}^{k}_{\alpha}(t)}\right)$, $0<\theta<1$, of FSPoK is given by 
\begin{equation*}
G^{k}_{\alpha}(\theta,t)=E_{\alpha,1}\left(-\left(k(\lambda_{1}+\lambda_{2})-\lambda_{1}\sum_{j=1}^{k}\theta^{j}-\lambda_{2}\sum_{j=1}^{k}\theta^{-j}\right)t^{\alpha}\right).
\end{equation*}	
\end{proposition}
\begin{proof}
The pgf of SPoK is given by (see Gupta {\it et al.} (2020), Eq. (39))
\begin{equation*}
G^{k}(\theta,t)=\exp\left(-t\left(k(\lambda_{1}+\lambda_{2})-\lambda_{1}\sum_{j=1}^{k}\theta^{j}-\lambda_{2}\sum_{j=1}^{k}\theta^{-j}\right)\right).
\end{equation*}	
Using the above result, we get	
\begin{align*}
G^{k}_{\alpha}(\theta,t)&=\int_{0}^{\infty}G^{k}(\theta,u)h_{\alpha}(u,t)\,\mathrm{d}u\\
&=\int_{0}^{\infty}\exp\left(-u\left(k(\lambda_{1}+\lambda_{2})-\lambda_{1}\sum_{j=1}^{k}\theta^{j}-\lambda_{2}\sum_{j=1}^{k}\theta^{-j}\right)\right)h_{\alpha}(u,t)\,\mathrm{d}u\\
&=E_{\alpha,1}\left(-\left(k(\lambda_{1}+\lambda_{2})-\lambda_{1}\sum_{j=1}^{k}\theta^{j}-\lambda_{2}\sum_{j=1}^{k}\theta^{-j}\right)t^{\alpha}\right).
\end{align*}
This completes the proof.	
\end{proof}
Using  the fact that the Mittag-Leffler function is an eigenfunction of the Caputo fractional derivative, it follows that
\begin{equation*}
\partial_{t}^{\alpha}G^{k}_{\alpha}(\theta,t)=-\left(k(\lambda_{1}+\lambda_{2})-\lambda_{1}\sum_{j=1}^{k}\theta^{j}-\lambda_{2}\sum_{j=1}^{k}\theta^{-j}\right)G^{k}_{\alpha}(\theta,t),\ \ G^{k}_{\alpha}(\theta,0)=1.
\end{equation*}
\begin{proposition}
The one-dimensional distributions of FSPoK $\{\mathcal{S}^{k}_{\alpha}(t)\}_{t\ge0}$ are not infinitely divisible.
\end{proposition}
\begin{proof}
From (\ref{skellam}) and (\ref{FSPoK}),  we get
\begin{align*}
\mathcal{S}^{k}_{\alpha}(t)&=N_{1}^{k}\left(Y_{\alpha}(t)\right)-N_{2}^{k}\left(Y_{\alpha}(t)\right)\\
&\overset{d}{=}N_{1}^{k}\left(t^{\alpha}Y_{\alpha}(1)\right)-N_{2}^{k}\left(t^{\alpha}Y_{\alpha}(1)\right),
\end{align*}
where $\overset{d}{=}$ means equal in distribution. Here, we have used the self-similarity property of	$\{Y_{\alpha}(t)\}_{t\ge0}$.
The following result holds for PPoK (see Sengar {\it et al.} (2020), Eq. (9)):
\begin{equation}\label{limit}
\frac{N^{k}(t)}{t}\to\frac{k(k+1)}{2}\lambda,\ \ \mathrm{in\  probability} \  \text{as}\ t\to\infty.
\end{equation}
Thus, $N^{k}(t)/t \to k(k+1)\lambda/2$, in distribution as $t\to \infty$. Therefore,
\begin{align*}
\lim_{t\to\infty}\frac{\mathcal{S}^{k}_{\alpha}(t)}{t^{\alpha}}&=\lim_{t\to\infty}\frac{N_{1}^{k}\left(t^{\alpha}Y_{\alpha}(1)\right)-N_{2}^{k}\left(t^{\alpha}Y_{\alpha}(1)\right)}{t^{\alpha}},\ \ \mathrm{in\  distribution}\\
&=Y_{\alpha}(1)\left(\lim_{t\to\infty}\frac{N_{1}^{k}\left(t^{\alpha}Y_{\alpha}(1)\right)}{t^{\alpha}Y_{\alpha}(1)}-\lim_{t\to\infty}\frac{N_{2}^{k}\left(t^{\alpha}Y_{\alpha}(1)\right)}{t^{\alpha}Y_{\alpha}(1)}\right)\\
&=\frac{k(k+1)}{2}(\lambda_{1}-\lambda_{2})Y_{\alpha}(1),\ \ \mathrm{in\  distribution}.
\end{align*}

It is known that $Y_{\alpha}(1)$ is not infinitely divisible (see Vellaisamy and Kumar (2018)). Let  $\mathcal{S}^{k}_{\alpha}(t)$ be infinitely divisible. Then, it follows that $\mathcal{S}^{k}_{\alpha}(t)/t^{\alpha}$ is  infinitely divisible (see Steutel and van
Harn (2004), Proposition 2.1). As the limit of a sequence of infinitely divisible random variables 
is infinitely divisible (see Steutel
and van Harn (2004), Proposition 2.2), it follows that  $Y_{\alpha}(1)$ is infinitely divisible. This leads to a contradiction.
\end{proof}

Let us assume that
\begin{equation}\label{r1r2}
r_{1}=k(k+1)(\lambda_{1}-\lambda_{2})/2\ \  \text{and}\ \  r_{2}=k(k+1)(2k+1)(\lambda_{1}+\lambda_{2})/6.
\end{equation}
The mean, variance and covariance function of $\{S^{k}(t)\}_{t\ge0}$ are given by (see Gupta {\it et al.} (2020))
\begin{equation}\label{meanvarskt}
\mathbb{E}\left(S^{k}(t)\right)=r_{1}t,\ \ \operatorname{Var}\left(S^{k}(t)\right)=r_{2}t \ \ \text{and}\ \  \operatorname{Cov}\left(S^{k}(s),S^{k}(t)\right)=r_{2}s,\ \ 0<s\leq t.
\end{equation}

The mean, variance and  covariance of FSPoK are
\begin{align}
\mathbb{E}\left(\mathcal{S}^{k}_{\alpha}(t)\right)&=r_{1}\mathbb{E}\left(Y_{\alpha}(t)\right),\label{es}\\
\operatorname{Var}\left(\mathcal{S}^{k}_{\alpha}(t)\right)&=r_{2}\mathbb{E}\left(Y_{\alpha}(t)\right)+r_{1}^{2}\operatorname{Var}\left(Y_{\alpha}(t)\right),\label{vs}\\	
\operatorname{Cov}\left(\mathcal{S}^{k}_{\alpha}(s),\mathcal{S}^{k}_{\alpha}(t)\right)&=r_{2}\mathbb{E}\left(Y_{\alpha}(s)\right)+r_{1}^{2}\operatorname{Cov}\left(Y_{\alpha}(s),Y_{\alpha}(t)\right),\label{cs}
\end{align}
which are obtained by using Theorem 2.1 of Leonenko {\it et al.} (2014).
\begin{remark}
On substituting $k=1$ in (\ref{es})-(\ref{cs}), we get the  mean, variance and covariance  of FSP (see Kerss {\it et al.} (2014), Remark 3.2).
\end{remark}
Next, we show that the FSPoK possesses the LRD property. The following definition  will be used (see D'Ovidio and Nane (2014), Maheshwari and Vellaisamy (2016)):
\begin{definition}
Let $s>0$ be fixed and $\{X(t)\}_{t\ge0}$ be a stochastic process such that
\begin{equation*}
\lim_{t\to\infty}\frac{\operatorname{Cov}(X(s),X(t))}{\sqrt{\operatorname{Var}(X(s))}\sqrt{\operatorname{Var}(X(t))}t^{-\gamma}}= c(s).
\end{equation*}
If $\gamma\in(0,1)$	the process $\{X(t)\}_{t\ge0}$ has the LRD property, and if $\gamma\in(1,2)$ it has the SRD property.
\end{definition}
\begin{theorem}
The	FSPoK has the LRD property.
\end{theorem}	
\begin{proof}
For a fixed $s>0$, we have 
\begin{equation}\label{mnt}
\lim_{t\to\infty}\frac{\operatorname{Cov}\left(\mathcal{S}^{k}_{\alpha}(s),\mathcal{S}^{k}_{\alpha}(t)\right)}{\sqrt{\operatorname{Var}\left(\mathcal{S}^{k}_{\alpha}(t)\right)}\sqrt{\operatorname{Var}\left(\mathcal{S}^{k}_{\alpha}(t)\right)}t^{-\alpha}}=c(s),
\end{equation}
where 
\begin{equation*}
c(s)=\left(\frac{1}{\Gamma(2\alpha)}  -  \frac{1}{ \alpha(\Gamma(\alpha))^2}\right)^{-1}
\left(\frac{\alpha r_{2} }{\Gamma(1+\alpha)  r_{1}^2}   +\frac{ \alpha s^{\alpha}}{\Gamma(1+2\alpha)}     \right).
\end{equation*}
The result in (\ref{mnt}) follows by using a result on p. $10$ of Leonenko {\it et al.} (2014), and the mean and variance of SPoK which are given in (\ref{meanvarskt}). Thus, the FSPoK exhibits the LRD property as $0<\alpha<1$.  
\end{proof}
\begin{remark}
For a fixed $h>0$, the increment of FSPoK is defined as
\begin{equation*}
Z^{k}_{h}(t)=\mathcal{S}^{k}_{\alpha}(t+h)-\mathcal{S}^{k}_{\alpha}(t),\ \ t\ge0.
\end{equation*}
It can be shown that the increment process $\{Z^{k}_{h}(t)\}_{t\ge0}$ exhibits the SRD property. The proof follows similar lines to that of Theorem 1 of Maheshwari and Vellaisamy (2016), and thus it is omitted.
\end{remark}
 
\section{The FSPoK time-changed by a L\'evy subordinator}
In this section, we consider a time-changed version of the FSPoK. We call it the time-changed  fractional Skellam process of order $k$ (TCFSPoK) and denote it by $\{\mathcal{Z}^{f}_{\alpha}
(t)\}_{t\ge0}$, $0<\alpha\le1$. It is defined as the  FSPoK time-changed by an independent L\'evy subordinator $\{D_f (t)\}_{t\ge0}$ with $\mathbb{E}\left(D^r_f(t)\right)<\infty$ for all $r>0$.  Thus,
\begin{equation}\label{qws11ww1}
\mathcal{Z}^{f}_{\alpha}(t)\coloneqq \mathcal{S}^{k}_{\alpha}(D_f (t)),\ \ t\ge0,
\end{equation}
where $\{\mathcal{S}^{k}_{\alpha}(t)\}_{t\ge0}$ is independent of $\{D_f (t)\}_{t\ge0}$. 

For $\alpha=1$, the TCFSPoK reduces to  a time-changed version of the SPoK (see Gupta {\it et al.} (2020)), that is, 
\begin{equation*}
\mathcal{Z}^{f}(t)\coloneqq \mathcal{S}^{k}(D_f (t)),\ \ t\ge0.
\end{equation*}
Its pmf $p^{f}(n,t)=\mathrm{Pr}\{\mathcal{Z}^{f}(t)=n\}$  is given by (see Gupta {\it et al.} (2020), Eq. (47))
\begin{equation*}
p^{f}(n,t)=\sum_{x=\max(0,-n)}^{\infty}\frac{(k\lambda_{1})^{n+x}(k\lambda_{2})^{x}}{(n+x)!x!}\mathbb{E}\left(e^{-k(\lambda_{1}+\lambda_{2})D_{f}(t)}D_{f}^{2n+x}(t)\right),\ \ n\in \mathbb{Z}.
\end{equation*}

The mean and covariance of inverse stable  subordinator are given by (see Leonenko {\it et al.} (2014), Eq. (8) and Eq. (10))
\begin{equation}\label{meani}
\mathbb{E}\left(Y_{\alpha}(t)\right)=\frac{t^{\alpha}}{\Gamma(\alpha+1)}
\end{equation}
and 
\begin{equation}\label{covin}
\operatorname{Cov}\left(Y_{\alpha}(s),Y_{\alpha}(t)\right)=\frac{1}{\Gamma^2(\alpha+1)}\left( \alpha s^{2\alpha}B(\alpha,\alpha+1)+F(\alpha;s,t)\right),\ \ 0<s\le t,
\end{equation}
where $F(\alpha;s,t)=\alpha t^{2\alpha}B(\alpha,\alpha+1;s/t)-(ts)^{\alpha}$. Here, $B(\alpha,\alpha+1)$  and $B(\alpha,\alpha+1;s/t)$ denote the beta function and the incomplete beta function, respectively.

Let $l_{1}=r_{1}/\Gamma(\alpha+1)$, $l_{2}=r_{2}/\Gamma(\alpha+1)$ and  $d=l_{1}^{2}\alpha B(\alpha,\alpha+1)$ where $r_{1}$ and $r_{2}$ are given by (\ref{r1r2}).

The mean of TCFSPoK is obtained as follows:
\begin{equation}\label{mean}
\mathbb{E}\left(\mathcal{Z}^{f}_{\alpha}(t)\right)=\mathbb{E}\left(\mathbb{E}\left(\mathcal{S}^{k}_{\alpha}(D_{f}(t))|D_{f}(t)\right)\right)=l_{1}\mathbb{E}\left(D_{f}^{\alpha}(t)\right).	
\end{equation}
On substituting  (\ref{meani}) and (\ref{covin}) in (\ref{cs}), we get
\begin{equation*}
\mathbb{E}\left(\mathcal{S}^{k}_{\alpha}(s)\mathcal{S}^{k}_{\alpha}(t)\right)=l_{2}s^{\alpha}+ds^{2\alpha}+\alpha l_{1}^{2}\left(t^{2\alpha}B(\alpha,\alpha+1;s/t)\right).
\end{equation*}
Thus,
\begin{align*}
\mathbb{E}\left(\mathcal{Z}^{f}_{\alpha}(s)\mathcal{Z}^{f}_{\alpha}(t)\right)&=\mathbb{E}\left(\mathbb{E}\left(\mathcal{S}^{k}_{\alpha}(D_{f}(s))\mathcal{S}^{k}_{\alpha}(D_{f}(t))|D_{f}(s),D_{f}(t)\right)\right)\\
&=l_{2}\mathbb{E}\left(D_{f}^{\alpha}(s)\right)+ d\mathbb{E}\left(D_{f}^{2\alpha}(s)\right) +\alpha l_{1}^{2}  \mathbb{E}\left(D_{f}^{2\alpha}(t)B\left(\alpha,\alpha+1;D_{f}(s)/D_{f}(t)\right)\right).
\end{align*}
Hence, the covariance of TCFSPoK is given by
\begin{align}\label{cov}
\operatorname{Cov}\left(\mathcal{Z}^{f}_{\alpha}(s),\mathcal{Z}^{f}_{\alpha}(t)\right)&=l_{2}\mathbb{E}\left(D_{f}^{\alpha}(s)\right)+ d\mathbb{E}\left(D_{f}^{2\alpha}(s)\right)-l_{1}^{2}\mathbb{E}\left(D_{f}^{\alpha}(s)\right)\mathbb{E}\left(D_{f}^{\alpha}(t)\right)\nonumber\\
&\ \ 
+ \alpha l_{1}^{2}  \mathbb{E}\left(D_{f}^{2\alpha}(t)B\left(\alpha,\alpha+1;D_{f}(s)/D_{f}(t)\right)\right).
\end{align}
Also, its variance is obtained by substituting $s=t$ in (\ref{cov}) which is given by
\begin{equation}\label{var}
\operatorname{Var}\left(\mathcal{Z}^{f}_{\alpha}(t)\right)=\mathbb{E}\left(D_{f}^{\alpha}(t)\right)\left(l_{2}-l_{1}^{2}\mathbb{E}\left(D_{f}^{\alpha}(t)\right)\right)+2d\mathbb{E}\left(D_{f}^{2\alpha}(t)\right).
\end{equation}
\begin{remark}
	On substituting $k=1$ and taking $\lambda_{2}\to 0$ in (\ref{mean})-(\ref{var}), we get the mean, covariance and variance of a time-changed version of the fractional Poisson process (see Maheshwari and Vellaisamy (2019), Theorem 3.2).
\end{remark}
Next we show that the TCFSPoK exhibits the LRD property under certain restrictions on the L\'evy subordinator. For positive functions $f$ and $g$, the notation  $f\sim g$ stands for $f(t)/g(t)\to 1$ as $t\to\infty$.

\begin{theorem}
Let $\{D_f(t)\}_{t\ge0}$ be a L\'evy subordinator with $f(\cdot)$ being its associated  Bern\v stein function defined in (\ref{fs}) such that $\mathbb{E}\left(D^r_f(t)\right)<\infty$ for all $r>0$. The TCFSPoK $\{\mathcal{Z}^{f}_{\alpha}(t)\}_{t\ge0}$, $0<\alpha<1$ exhibits the LRD property if 
\begin{equation}\label{12}
\mathbb{E}\left(D_f^{i\alpha}(t)\right)\sim k_it^{i\rho},\ \ i=1,2,
\end{equation}
for some  $0< \rho <1$,  and positive  constants $k_1$ and $k_2$ such that $k_2\geq k_1^2$. 
\end{theorem}	
\begin{proof}
Let $0<s<t$. In (\ref{cov}), we use the following asymptotic result for large $t$ (see Maheshwari and Vellaisamy (2019), Theorem 3.3)
\begin{equation*}
\alpha\mathbb{E}\left(D_f^{2\alpha}(t)B\left(\alpha,\alpha+1;D_f(s)/D_f(t)\right)\right)\sim \mathbb{E}\left(D_f^{\alpha}(s)\right)\E\left(D_f^{\alpha}(t-s)\right),
\end{equation*}
to obtain
\begin{align*}
\operatorname{Cov}\left(\mathcal{Z}^{f}_{\alpha}(s),\mathcal{Z}^{f}_{\alpha}(t)\right)&\sim l_{2}\mathbb{E}\left(D_{f}^{\alpha}(s)\right)+ d\mathbb{E}\left(D_{f}^{2\alpha}(s)\right)-l_{1}^{2}\mathbb{E}\left(D_{f}^{\alpha}(s)\right)\mathbb{E}\left(D_{f}^{\alpha}(t)\right)\\
&\ \ \ 
+l_{1}^{2}  \mathbb{E}\left(D_f^{\alpha}(s)\right)\E\left(D_f^{\alpha}(t-s)\right)\\
&\sim l_{2}\mathbb{E}\left(D_{f}^{\alpha}(s)\right)+ d\mathbb{E}\left(D_{f}^{2\alpha}(s)\right)-l_{1}^{2}  \mathbb{E}\left(D_f^{\alpha}(s)\right)k_1(t^\rho-(t-s)^\rho)\\
&\sim l_{2}\mathbb{E}\left(D_{f}^{\alpha}(s)\right)+ d\mathbb{E}\left(D_{f}^{2\alpha}(s)\right)-l_{1}^{2}  \mathbb{E}\left(D_f^{\alpha}(s)\right)k_1s\rho t^{\rho-1},
\end{align*}	
where we have used (\ref{12}) in the penultimate step. Again by using (\ref{12}) in (\ref{var}), we get
\begin{align*}
\operatorname{Var}\left(\mathcal{Z}^{f}_{\alpha}(t)\right)&\sim l_{2}k_1t^\rho- l_{1}^{2}k^{2}_1t^{2\rho} +2dk_2t^{2\rho}\\
&\sim \left(2dk_2-k^{2}_1l_{1}^{2}\right)t^{2\rho}\\
&=\frac{r_{1}^2}{\alpha}\left(\frac{k_2}{\Gamma(2\alpha)}-\frac{k_1^2}{\alpha\Gamma^2(\alpha)}\right)t^{2\rho}.
\end{align*}
\noindent Thus, for large $t$, we have
\begin{align*}
\operatorname{Corr}\left(\mathcal{Z}^{f}_{\alpha}(s),\mathcal{Z}^{f}_{\alpha}(t)\right)&\sim\frac{l_{2}\mathbb{E}\left(D_{f}^{\alpha}(s)\right)+ d\mathbb{E}\left(D_{f}^{2\alpha}(s)\right)-l_{1}^{2}  \mathbb{E}\left(D_f^{\alpha}(s)\right)k_1s\rho t^{\rho-1}}{\sqrt{\operatorname{Var}\left(\mathcal{Z}^{f}_{\alpha}(s)\right)}\sqrt{\left(2dk_2-k^{2}_1l_{1}^{2}\right)t^{2\rho}}}\\
&\sim c_{1}(s)t^{-\rho},  		
\end{align*}
\noindent where
\begin{equation*}
c_{1}(s)=\frac{l_{2}\mathbb{E}\left(D_{f}^{\alpha}(s)\right)+ d\mathbb{E}\left(D_{f}^{2\alpha}(s)\right)}{\sqrt{\operatorname{Var}\left(\mathcal{Z}^{f}_{\alpha}(s)\right)\left(2dk_2-k^{2}_1l_{1}^{2}\right)}}.
\end{equation*}		
Hence, the TCFSPoK exhibits the LRD property as  $0< \rho <1$.
\end{proof}
\begin{remark}
In a similar way it can be shown that $\{\mathcal{Z}^{f}(t)\}_{t\ge0}$ exhibits the LRD property.
\end{remark}

The following result will be used to prove the law of iterated logarithm (LIL) for TCFSPoK (see Bertoin (1996), Theorem 14, p. 92).
\begin{lemma}
Let $\{D_{f}(t)\}_{t\ge0}$ be a L\'evy subordinator whose associated Bern\v stein function  $f$  is regularly varying at $0+$ with index $0<\gamma<1$, {\it i.e.},  $\lim_{x\rightarrow 0+}f(\lambda x)/f(x)=\lambda^\gamma$, $\lambda>0$. Also, let
\begin{equation}\label{gt}
g(t)=\frac{\log\log t}{\phi(t^{-1}\log\log t)},\ \ t>e,
\end{equation}
where $\phi$ is the inverse of $f$. Then, 
\begin{equation}\label{qweqw1}
\liminf_{t\to\infty}\frac{D_{f}(t)}{g(t)}=\gamma(1-\gamma)^{(1-\gamma)/\gamma},\ \ \text{a.s.}
\end{equation}
\end{lemma}

\begin{theorem}
Let the Bern\v stein function $f(\cdot)$ associated with L\'evy subordinator $\{D_{f}(t)\}_{t\ge0}$ be regularly varying at 0+ with index $0<\gamma<1$. Then, 
\begin{equation}\label{lil}
\liminf_{t\rightarrow\infty}\frac{\mathcal{Z}^{f}_{\alpha}(t)}{(g(t))^{\alpha}}\stackrel{d}{=}\frac{k(k+1)}{2}(\lambda_{1}-\lambda_{2}) Y_{\alpha}(1)\gamma^\alpha\left(1-\gamma\right)^{\alpha(1-\gamma)/\gamma},
\end{equation}
where $	g(t)$ is given in (\ref{gt}).
	
\end{theorem}
\begin{proof}
Note that $D_f(t)\to\infty$ as $t\to\infty$, {\it a.s.} (see Bertoin (1996), p. 73). From (\ref{skellam}), (\ref{FSPoK}) and  (\ref{qws11ww1}), we get
\begin{align*}
\mathcal{Z}^{f}_{\alpha}(t)&=N_{1}^{k}(Y_{\alpha}(D_{f}(t)))-N_{2}^{k}(Y_{\alpha}(D_{f}(t)))\\
&\stackrel{d}{=}N_{1}^{k}(D^{\alpha}_{f}(t)Y_{\alpha}(1))-N_{2}^{k}(D^{\alpha}_{f}(t)Y_{\alpha}(1)),
\end{align*}  
where we have used the self-similarity property of	$\{Y_{\alpha}(t)\}_{t\ge0}$.
Hence,
\begin{align*}
\liminf_{t\rightarrow\infty}\frac{\mathcal{Z}^{f}_{\alpha}(t)}{(g(t))^{\alpha}}&\stackrel{d}{=}\liminf_{t\rightarrow\infty}\frac{N_{1}^{k}(D^{\alpha}_{f}(t)Y_{\alpha}(1))-N_{2}^{k}(D^{\alpha}_{f}(t)Y_{\alpha}(1))}{(g(t))^{\alpha}}\\
&=\liminf_{t\rightarrow\infty}\left(\frac{N_{1}^{k}(D^{\alpha}_{f}(t)Y_{\alpha}(1))-N_{2}^{k}(D^{\alpha}_{f}(t)Y_{\alpha}(1))}{D^{\alpha}_{f}(t)Y_{\alpha}(1)}\right)\frac{D^{\alpha}_{f}(t)Y_{\alpha}(1)}{(g(t))^{\alpha}}\\
&\stackrel{d}{=}\frac{k(k+1)}{2}(\lambda_{1}-\lambda_{2}) Y_{\alpha}(1)\left(\liminf_{t\rightarrow\infty}\frac{D_{f}(t)}{g(t)}\right)^{\alpha},\ \ (\text{using}\ (\ref{limit}))\\
&\stackrel{d}{=}\frac{k(k+1)}{2}(\lambda_{1}-\lambda_{2}) Y_{\alpha}(1)\gamma^\alpha\left(1-\gamma\right)^{\alpha(1-\gamma)/\gamma},
\end{align*}
where the last step follows from \eqref{qweqw1}.
\end{proof}
\begin{remark}
	On substituting $k=1$ and taking $\lambda_{2}\to 0$ in (\ref{lil}), we get the LIL for a time-changed version of the fractional Poisson process (see Maheshwari and Vellaisamy (2019), Theorem 3.5).
\end{remark}
\subsection{Some special cases of the TCFSPoK}
In this subsection, we time-change the FSPoK and the SPoK by three specific L\'evy subordinators, namely, the gamma subordinator, the tempered stable subordinator and the inverse Gaussian subordinator. We obtain the  governing systems of differential equations for their one-dimensional distributions.
\subsubsection{FSPoK time-changed by gamma subordinator}
Let $\{Z(t)\}_{t\ge0}$ be a gamma subordinator with the following probability density function (pdf):
\begin{equation*}
g(x,t)=\frac{a^{bt}}{\Gamma(bt)}x^{bt-1}e^{-ax},\ \ x>0,
\end{equation*}
where $a>0$ and $b>0$. The Bern\v stein function $f_{1}(s)$ associated  with $\{Z(t)\}_{t\ge0}$ is given by $f_{1}(s)=b\log(1+s/a),\ s>0$ (see Applebaum (2009), p. 55).

The FSPoK time-changed by an independent gamma subordinator is defined as 
\begin{equation}\label{gam}
\mathcal{Z}^{f_{1}}_{\alpha}(t)\coloneqq \mathcal{S}^{k}_{\alpha}(Z(t)),\ \ t\ge0.
\end{equation}
The following result will be used to obtain the governing system of differential equations for its pmf $p_{\alpha}^{f_{1}}(n,t)=\mathrm{Pr}\{\mathcal{Z}^{f_{1}}_{\alpha}(t)=n\}$, $n\in\mathbb{Z}$.
\begin{lemma}[Vellaisamy and Maheshwari (2018)]\label{lemma}
For any $\gamma\ge1$, the pdf $g(x,t)$ of gamma subordinator solves
\begin{align*}
D_t^{\gamma}g(x,t)&=bD_t^{\gamma-1}\left(\log(a x)-\psi(bt)\right)g(x,t), \ \ \ x>0,\\
g(x,0)&=0.	
\end{align*}
Here, $\psi(x)\coloneqq\Gamma^{\prime}(x)/\Gamma(x)$ is the digamma function and $D_t^{\gamma}$ is the R-L fractional derivative defined in (\ref{RL}).
\end{lemma}
\begin{theorem}
Let $\gamma\ge1$ and $\psi(x)$ be the digamma function. Then, the pmf $p_{\alpha}^{f_{1}}(n,t)$ solves the following equation:
\begin{equation*}
D_t^{\gamma}p_{\alpha}^{f_{1}}(n,t)=bD_t^{\gamma-1}\left(\log (a)-\psi(bt)\right)p_{\alpha}^{f_{1}}(n,t)+b\int_{0}^{\infty}p^{k}_{\alpha}(n,x)\log (x)D_t^{\gamma-1}g(x,t)\,\mathrm{d}x.
\end{equation*}
\end{theorem}
\begin{proof}
From (\ref{gam}), we have
\begin{equation}\label{qaza122}
p_{\alpha}^{f_{1}}(n,t)=\int_{0}^{\infty}p^{k}_{\alpha}(n,x)g(x,t)\,\mathrm{d}x.
\end{equation}
Taking the R-L fractional derivative in (\ref{qaza122}) and using Lemma \ref{lemma}, we get
\begin{align*}
D_t^{\gamma}p_{\alpha}^{f_{1}}(n,t)&=\int_{0}^{\infty}p^{k}_{\alpha}(n,x)D_t^{\gamma}g(x,t)\,\mathrm{d}x\\
&=b\int_{0}^{\infty}p^{k}_{\alpha}(n,x)D_t^{\gamma-1}\left(\log(a x)-\psi(bt)\right)g(x,t)\,\mathrm{d}x\\
&=bD_t^{\gamma-1}\log (a) \int_{0}^{\infty}p^{k}_{\alpha}(n,x)g(x,t)\,\mathrm{d}x+b \int_{0}^{\infty} p^{k}_{\alpha}(n,x)\log (x)D_t^{\gamma-1}g(x,t)\,\mathrm{d}x\\
&\ \ - bD_t^{\gamma-1}\psi(bt)\int_{0}^{\infty}p^{k}_{\alpha}(n,x)g(x,t)\,\mathrm{d}x\\
&=bD_t^{\gamma-1}\left(\log (a)-\psi(bt)\right)p_{\alpha}^{f_{1}}(n,t)+b\int_{0}^{\infty}p^{k}_{\alpha}(n,x)\log (x)D_t^{\gamma-1}g(x,t)\,\mathrm{d}x.
\end{align*}
This completes the proof.
\end{proof}
Next we discuss two particular cases of the time-changed SPoK (see Gupta {\it et al.} (2020)).
\subsubsection{SPoK time-changed by tempered stable subordinator}
Let $\{\mathscr{D}_{\eta,\nu}(t)\}_{t\ge0}$  denote the tempered stable subordinator (TSS) with stability index $0<\nu<1$ and the  tempering parameter $\eta>0$. The  Bern\v stein function $f_{2}(s)$ associated with TSS is given by 
\begin{equation}\label{bs1}
f_{2}(s)=(\eta+s)^{\nu}-\eta^{\nu},\ s>0.
\end{equation}
The SPoK time-changed by an independent TSS is defined as 
\begin{equation}\label{tss}
\mathcal{Z}^{f_{2}}(t)\coloneqq \mathcal{S}^{k}(\mathscr{D}_{\eta,\nu}(t)),\ \ t\ge0.
\end{equation}	
\begin{proposition}
The pmf $p^{f_{2}}(n,t)=\mathrm{Pr}\{\mathcal{Z}^{f_{2}}(t)=n\}$, $n\in\mathbb{Z}$, is the solution of the following differential equation:
\begin{equation*}
\left(\eta^{\nu}-\frac{\mathrm{d}}{\mathrm{d} t}\right)^{1/\nu}p^{f_{2}}(n,t)=\eta p^{f_{2}}(n,t)+k(\lambda_{1}+\lambda_{2})p^{f_{2}}(n,t)-\lambda_{1}\sum_{j=1}^{k}p^{f_{2}}(n-j,t)-\lambda_{2}\sum_{j=1}^{k}p^{f_{2}}(n+j,t).
\end{equation*}
\end{proposition}
\begin{proof}
From (\ref{tss}), we have
\begin{equation*}
p^{f_{2}}(n,t)=\int_{0}^{\infty}p^{k}(n,x)	h_{\eta,\nu}(x,t)\,\mathrm{d}x.
\end{equation*}
Here, $h_{\eta,\nu}(x,t)$ denotes the pdf of TSS. It is known that  (See Beghin (2015), Eq. (15))
\begin{equation*}
\frac{\partial}{\partial x}h_{\eta,\nu}(x,t)=-\eta h_{\eta,\nu}(x,t)+\left(\eta^{\nu}-\frac{\partial}{\partial t}\right)^{1/\nu}h_{\eta,\nu}(x,t),
\end{equation*}
with initial conditions $h_{\eta,\nu}(x,0)=\delta_0(x)$ and $h_{\eta,\nu}(0,t)=0$. Here, $\delta_0(x)$ is the Dirac delta function. Using the following results: $\lim_{x\to 0}h_{\eta,\nu}(x,t)=\lim_{x\to \infty}h_{\eta,\nu}(x,t)=0$, we get 
\begin{align*}
\left(\eta^{\nu}-\frac{\mathrm{d}}{\mathrm{d} t}\right)^{1/\nu}p^{f_{2}}(n,t)&=\int_{0}^{\infty}p^{k}(n,x)\left(\frac{\partial}{\partial x}	h_{\eta,\nu}(x,t)+\eta h_{\eta,\nu}(x,t)\right)\,\mathrm{d}x\\
&=\eta p^{f_{2}}(n,t)-\int_{0}^{\infty}h_{\eta,\nu}(x,t)\frac{\mathrm{d}}{\mathrm{d} x}p^{k}(n,x)	\,\mathrm{d}x\\
&=\eta p^{f_{2}}(n,t)- \int_{0}^{\infty}\bigg(-k(\lambda_{1}+\lambda_{2})p^{k}(n,x)+\lambda_{1}\sum_{j=1}^{k}p^{k}(n-j,x)\\
&\ \ +\lambda_{2}\sum_{j=1}^{k}p^{k}(n+j,x)\bigg)h_{\eta,\nu}(x,t)\,\mathrm{d}x,\ \ {(\text{using (\ref{skellam gov})})}\\
&=\eta p^{f_{2}}(n,t)+k(\lambda_{1}+\lambda_{2})p^{f_{2}}(n,t)\\
&\ \ -\lambda_{1}\sum_{j=1}^{k}p^{f_{2}}(n-j,t)-\lambda_{2}\sum_{j=1}^{k}p^{f_{2}}(n+j,t).
\end{align*}
This completes the proof.
\end{proof}
If $\nu^{-1}=m\ge2$  is an integer then the pmf $p^{f_{2}}(n,t)$ solves 
\begin{equation*}
\sum_{k=1}^{m}(-1)^{k}\binom{m}{k}\eta^{(1-k/m)}\frac{\mathrm{d}^{k}}{\mathrm{d} t^{k}}p^{f_{2}}(n,t)=k(\lambda_{1}+\lambda_{2})p^{f_{2}}(n,t)-\lambda_{1}\sum_{j=1}^{k}p^{f_{2}}(n-j,t)-\lambda_{2}\sum_{j=1}^{k}p^{f_{2}}(n+j,t).
\end{equation*}

\subsubsection{SPoK time-changed by inverse Gaussian subordinator}
Let $\{Y(t)\}_{t\ge0}$ be an inverse Gaussian (IG) subordinator with the following pdf (see Applebaum (2009), Eq. (1.27))
\begin{equation*}
q(x,t)=(2\pi)^{-1/2}\delta tx^{-3/2}e^{\delta\gamma t-\frac{1}{2}(\delta^{2}t^{2}x^{-1}+\gamma^{2}x)}, \ \ x>0,\ \delta>0,\ \gamma>0
\end{equation*}
and the associated Bern\v stein function 
\begin{equation}\label{ploiuy67}
f_{3}(s)=\delta\left(\sqrt{2s+\gamma^2}-\gamma\right),\ \ s>0.
\end{equation}
The SPoK time-changed by an independent IG subordinator is defined as 
\begin{equation}\label{ig}
\mathcal{Z}^{f_{3}}(t)\coloneqq \mathcal{S}^{k}(Y(t)),\ \ t\ge0.
\end{equation}
\begin{proposition}\label{p4.1}
The pmf  $p^{f_{3}}(n,t)=\mathrm{Pr}\{\mathcal{Z}^{f_{3}}(t)=n\}$, $n\in\mathbb{Z}$, solves the following differential equation:
\begin{align*}
\left(\frac{\mathrm{d}^2}{\mathrm{d}t^2}-2\delta \gamma \frac{\mathrm{d}}{\mathrm{d}t}\right) p^{f_{3}}(n,t)&=2\delta^2\Bigg(k(\lambda_{1}+\lambda_{2})p^{f_{3}}(n,t)\\
&\hspace*{3cm} -\lambda_{1}\sum_{j=1}^{k}p^{f_{3}}(n-j,t)-\lambda_{2}\sum_{j=1}^{k}p^{f_{3}}(n+j,t)\Bigg).
\end{align*}
\end{proposition}
\begin{proof}
From (\ref{ig}), we have
\begin{equation}\label{123}
p^{f_{3}}(n,t)=\int_0^{\infty}p^{k}(n,x)q(x,t)\,\mathrm{d}x.
\end{equation}	
On taking derivatives, we get
\begin{equation*}
\frac{\mathrm{d}}{\mathrm{d}t} p^{f_{3}}(n,t) = \int_0^{\infty}p^{k}(n,x) \frac{\partial}{\partial t} q(x,t)\,\mathrm{d}x
\end{equation*}
and
\begin{equation*}
\frac{\mathrm{d}^2}{\mathrm{d} t^2} p^{f_{3}}(n,t) = \int_0^{\infty} p^{k}(n,x) \frac{\partial^2}{\partial t^2} q(x,t)\,\mathrm{d}x. 
\end{equation*}
The following results hold for $q(x,t)$  (see Vellaisamy and Kumar (2018), Eq. (3.3))
\begin{equation}\label{qxt212}
\frac{\partial^2}{\partial t^2}q(x,t)-2\delta \gamma \frac{\partial}{\partial t}q(x,t) = 2\delta^2\frac{\partial}{\partial x}q(x,t)
\end{equation}
and $\lim_{x\to\infty}q(x,t)=\lim_{x\to 0}q(x,t)=0$.
From (\ref{123}), we get
\begin{align*}
\left(\frac{\mathrm{d}^2}{\mathrm{d} t^2}-2\delta \gamma \frac{\mathrm{d}}{\mathrm{d} t}\right) p^{f_{3}}(n,t)
&=\int_0^{\infty}p^{k}(n,x) \left(\frac{\partial^2}{\partial t^2}-2\delta \gamma \frac{\partial}{\partial t}\right) q(x,t)\,\mathrm{d}x\\
&= 2\delta^2 \int_0^{\infty} p^{k}(n,x)\frac{\partial}{\partial x} q(x,t)\,\mathrm{d}x, \ \ (\text{using (\ref{qxt212})})\\ 
&= -2\delta^2 \int_0^{\infty} q(x,t)\frac{\mathrm{d}}{\mathrm{d}x} p^{k}(n,x)\, \mathrm{d}x\label{jahxa11}\\
&=-2\delta^2\int_0^{\infty}\bigg(-k(\lambda_{1}+\lambda_{2})p^{k}(n,x)+\lambda_{1}\sum_{j=1}^{k}p^{k}(n-j,x)\\
&\ \ +\lambda_{2}\sum_{j=1}^{k}p^{k}(n+j,x)\bigg)q(x,t)\,\mathrm{d}x, \ \ (\text{using (\ref{skellam gov})})\\  
&=2\delta^2\Bigg(k(\lambda_{1}+\lambda_{2})p^{f_{3}}(n,t)-\lambda_{1}\sum_{j=1}^{k}p^{f_{3}}(n-j,t)\\
&\hspace*{4cm}-\lambda_{2}\sum_{j=1}^{k}p^{f_{3}}(n+j,t)\Bigg).
\end{align*}
This completes the proof. 
\end{proof}

\section{The FSPoK time-changed by inverse subordinator}
The first hitting time  of subordinator $\{D_f (t)\}_{t\ge0}$ is called the inverse subordinator. It is defined as
\begin{equation*}
H_f (t)\coloneqq\inf\{r\ge0: D_f (r)> t\}, \ \ t\ge0.
\end{equation*}
It is known that $\mathbb{E}\left(H^r_f(t)\right)<\infty$ for all $r>0$ (see Aletti {\it et al.} (2018), Section 2.1). 

Let 
\begin{equation}\label{qws11ww91}
\bar{\mathcal{Z}}^{f}_{\alpha}(t)\coloneqq \mathcal{S}^{k}_{\alpha}(H_f (t)),\ \ t\ge0,
\end{equation}
be the FSPoK time-changed by an independent inverse subordinator. 

For $\alpha=1$, the process defined in (\ref{qws11ww91}) reduces to  a time-changed version of the SPoK, that is, 
\begin{equation*}
\bar{\mathcal{Z}}^{f}(t)\coloneqq \mathcal{S}^{k}(H_f (t)),\ \ t\ge0.
\end{equation*}
The pmf $\bar{p}^{f}(n,t)=\mathrm{Pr}\{\bar{\mathcal{Z}}^{f}(t)=n\}$ of $\{\bar{\mathcal{Z}}^{f}(t)\}_{t\ge0}$ is given by 
\begin{equation*}
\bar{p}^{f}(n,t)=\sum_{x=\max(0,-n)}^{\infty}\frac{(k\lambda_{1})^{n+x}(k\lambda_{2})^{x}}{(n+x)!x!}\mathbb{E}\left(e^{-k(\lambda_{1}+\lambda_{2})H_{f}(t)}H_{f}^{2n+x}(t)\right),\ \ n\in \mathbb{Z}.
\end{equation*}
The proof follows along the similar lines to that of Theorem 2 of Gupta {\it et al.} (2020). 

The mean, variance and covariance of $\{\bar{\mathcal{Z}}^{f}_{\alpha}(t)\}_{t\ge0}$ are given as follows:
	
Let $l_{1}=r_{1}/\Gamma(\alpha+1)$, $l_{2}=r_{2}/\Gamma(\alpha+1)$ and  $d=l_{1}^{2}\alpha B(\alpha,\alpha+1)$ where $B(\alpha,\alpha+1)$ is the beta function, and  $r_{1}$ and $r_{2}$ are given by (\ref{r1r2}). Then,\vspace*{.2cm}
		
\noindent $(i)\    \ \mathbb{E}\left(\bar{\mathcal{Z}}^{f}_{\alpha}(t)\right)=l_{1}\mathbb{E}\left(H_{f}^{\alpha}(t)\right)$,\vspace*{.1cm}\\
\noindent $(ii)\ \operatorname{Var}\left(\bar{\mathcal{Z}}^{f}_{\alpha}(t)\right)=\mathbb{E}\left(H_{f}^{\alpha}(t)\right)\left(l_{2}-l_{1}^{2}\mathbb{E}\left(H_{f}^{\alpha}(t)\right)\right)+2d\mathbb{E}\left(H_{f}^{2\alpha}(t)\right)$,\vspace*{.1cm}\\
\noindent $(iii)\ \operatorname{Cov}\left(\bar{\mathcal{Z}}^{f}_{\alpha}(s),\bar{\mathcal{Z}}^{f}_{\alpha}(t)\right)=l_{2}\mathbb{E}\left(H_{f}^{\alpha}(s)\right)+ d\mathbb{E}\left(H_{f}^{2\alpha}(s)\right)-l_{1}^{2}\mathbb{E}\left(H_{f}^{\alpha}(s)\right)\mathbb{E}\left(H_{f}^{\alpha}(t)\right)$\vspace*{.1cm}\\
$\hspace*{4.6cm}
+ \alpha l_{1}^{2}  \mathbb{E}\left(H_{f}^{2\alpha}(t)B\left(\alpha,\alpha+1;H_{f}(s)/H_{f}(t)\right)\right)$,
	
\noindent where $0<s\le t$ and $B(\alpha,\alpha+1;H_{f}(s)/H_{f}(t))$ is the incomplete beta function.

The proof of  $(i)$-$(iii)$ follows similar lines to the corresponding results of $\{\mathcal{Z}^{f}_{\alpha}(t)\}_{t\ge0}$ given in the previous section.

Next we discuss two particular cases of the time-changed process $\{\bar{\mathcal{Z}}^{f}(t)\}_{t\ge0}$.
\subsection{SPoK time-changed by the inverse TSS}
Let $\{\mathscr{L}_{\eta,\nu}(t)\}_{t\ge0}$ denote the inverse TSS which is defined as the first hitting time of TSS $\{\mathscr{D}_{\eta,\nu}(t)\}_{t\ge0}$, $0<\nu<1$, $\eta>0$. That is,
\begin{equation*}
\mathscr{L}_{\eta,\nu}(t)= \inf \{ s\ge 0 : \mathscr{D}_{\eta,\nu}(s)>t\},\ \ t\ge0.
\end{equation*}
The SPoK time-changed by an independent inverse TSS is defined as  
\begin{equation}\label{inv1}
\bar{\mathcal{Z}}^{f_{2}}(t)\coloneqq \mathcal{S}^{k}(\mathscr{L}_{\eta,\nu} (t)),\ \ t\ge0,
\end{equation}
where the associated Bern\v stein function $f_{2}$  is given in (\ref{bs1}).
\begin{proposition}
For $n\in \mathbb{Z}$, the pmf $\bar{p}^{f_{2}}(n,t)=\mathrm{Pr}\{\bar{\mathcal{Z}}^{f_2}(t)=n\}$ solves the following differential equation:
\begin{align*}
\left(\eta+\frac{\mathrm{d}}{\mathrm{d} t}\right)^{\nu}\bar{p}^{f_{2}}(n,t)&=p^{k}(n,x)l_{\eta,\nu}(x,t)\big|_{x=0}-k(\lambda_{1}+\lambda_{2})\bar{p}^{f_{2}}(n,t)+\lambda_{1}\sum_{j=1}^{k}\bar{p}^{f_{2}}(n-j,t)\\
&\ \ +\lambda_{2}\sum_{j=1}^{k}\bar{p}^{f_{2}}(n+j,t)+\eta^{\nu}\bar{p}^{f_{2}}(n,t)-t^{-\nu}E^{1-\nu}_{1,1-\nu}(-\eta t)p^{k}(n,0),
\end{align*}
where $E^{1-\nu}_{1,1-\nu}(\cdot)$ is the three-parameter Mittag-Leffler function defined in (\ref{mit}).
\end{proposition}
\begin{proof}
From (\ref{inv1}), we get
\begin{equation}\label{gfdt}
\bar{p}^{f_{2}}(n,t)=\int_{0}^{\infty}p^{k}(n,x)l_{\eta,\nu}(x,t)\,\mathrm{d}x, \ \ t\ge0,
\end{equation}
where $p^{k}(n,x)$ and $l_{\eta,\nu}(x,t)$ are the pmf of SPoK and the pdf of $\mathscr{L}_{\eta,\nu} (t)$, respectively.
		 
The following result will be used (See Kumar {\it et al.} (2019), Eq. (25)):
\begin{equation*}
\frac{\partial}{\partial x}l_{\eta,\nu}(x,t)=-\left(\eta+\frac{\partial}{\partial t}\right)^{\nu}l_{\eta,\nu}(x,t)+\eta^{\nu} l_{\eta,\nu}(x,t)-t^{-\nu}E^{1-\nu}_{1,1-\nu}(-\eta t)\delta_0(x),
\end{equation*}
where $\delta_0(x)=l_{\eta,\nu}(x,0)$. Also, we have $\lim_{x\to \infty}l_{\eta,\nu}(x,t)=0$ (see Alrawashdeh {\it et al.} (2017), Lemma 4.6). Using the above results in (\ref{gfdt}), we get
\begin{align*}
\left(\eta+\frac{\mathrm{d}}{\mathrm{d} t}\right)^{\nu}\bar{p}^{f_{2}}(n,t)&=\int_{0}^{\infty}p^{k}(n,x)\left(-\frac{\partial}{\partial x}l_{\eta,\nu}(x,t)+\eta^{\nu} l_{\eta,\nu}(x,t)-t^{-\nu}E^{1-\nu}_{1,1-\nu}(-\eta t)\delta_0(x)\right)\,\mathrm{d}x\\
&=p^{k}(n,x)l_{\eta,\nu}(x,t)\big|_{x=0}+\int_{0}^{\infty}l_{\eta,\nu}(x,t)\frac{\mathrm{d}}{\mathrm{d}x}p^{k}(n,x)\,\mathrm{d}x+\eta^{\nu}\bar{p}^{f_{2}}(n,t)\\
&\ \ -t^{-\nu}E^{1-\nu}_{1,1-\nu}(-\eta t)\int_{0}^{\infty}p^{k}(n,x)\delta_0(x)\,\mathrm{d}x\\
&=p^{k}(n,x)l_{\eta,\nu}(x,t)\big|_{x=0}+\int_{0}^{\infty}\bigg(-k(\lambda_{1}+\lambda_{2})p^{k}(n,x)+\lambda_{1}\sum_{j=1}^{k}p^{k}(n-j,x)\nonumber\\
&\ \ +\lambda_{2}\sum_{j=1}^{k}p^{k}(n+j,x)\bigg)l_{\eta,\nu}(x,t)\,\mathrm{d}x+\eta^{\nu}\bar{p}^{f_{2}}(n,t)-t^{-\nu}E^{1-\nu}_{1,1-\nu}(-\eta t)p^{k}(n,0)\\
&=p^{k}(n,x)l_{\eta,\nu}(x,t)\big|_{x=0}-k(\lambda_{1}+\lambda_{2})\bar{p}^{f_{2}}(n,t)+\lambda_{1}\sum_{j=1}^{k}\bar{p}^{f_{2}}(n-j,t)\\
&\ \ +\lambda_{2}\sum_{j=1}^{k}\bar{p}^{f_{2}}(n+j,t)+\eta^{\nu}\bar{p}^{f_{2}}(n,t)-t^{-\nu}E^{1-\nu}_{1,1-\nu}(-\eta t)p^{k}(n,0).
\end{align*}
This gives the required result. 
\end{proof}
\subsection{SPoK time-changed by the first hitting time of IG subordinator}
Consider an IG subordinator whose associated Bern\v stein function $f_{3}$ is given in (\ref{ploiuy67}). The first hitting time $\{H(t)\}_{t\ge0}$ of the IG subordinator $\{Y(t)\}_{t\geq0}$ is  defined as
\begin{equation*}
H(t)\coloneqq \inf \{ s\ge 0 : Y(s)>t\},\ \ t\geq0.
\end{equation*}
We define the following time-changed process:
\begin{equation}\label{inv}
\bar{\mathcal{Z}}^{f_{3}}(t)\coloneqq \mathcal{S}^{k}(H (t)),\ \ t\ge0,
\end{equation}
where the SPoK $\{\mathcal{S}^{k}(t)\}_{t\geq0}$ is independent of $\{H(t)\}_{t\ge0}$.
\begin{proposition}
For $n\in \mathbb{Z}$,	the pmf $\bar{p}^{f_3}(n,t)=\mathrm{Pr}\{\bar{\mathcal{Z}}^{f_{3}}(t)=n\}$  solves the following differential equation:
\begin{align*}
\delta\left(\gamma^{2}+2\frac{\mathrm{d}}{\mathrm{d} t}\right)^{1/2}\bar{p}^{f_{3}}(n,t)&=\left(\delta\gamma-k(\lambda_{1}+\lambda_{2})\right)\bar{p}^{f_{3}}(n,t)+\lambda_{1}\sum_{j=1}^{k}\bar{p}^{f_{3}}(n-j,t)\\
&\ \ +\lambda_{2}\sum_{j=1}^{k}\bar{p}^{f_{3}}(n+j,t) -\delta\gamma \mathrm{Erf}\left(\frac{\gamma\sqrt{t}}{\sqrt{2}}\right)p^{k}(n,0),
\end{align*}
where $\mathrm{Erf}(\cdot)$ is the error function.
\end{proposition}
\begin{proof}
From (\ref{inv}), we have
\begin{equation}\label{gfdt1}
\bar{p}^{f_{3}}(n,t)=\int_{0}^{\infty}p^{k}(n,x)h(x,t)\,\mathrm{d}x, \ \ t\ge0,
\end{equation}
where $h(x,t)$ is the pdf of $H(t)$.

The following result will be used (see Wyloma\'nska {\it et al.} (2016), Eq. (2.22))
\begin{equation*}
\frac{\partial}{\partial x}h(x,t)=-\delta\left(\gamma^{2}+2\frac{\partial}{\partial t}\right)^{1/2}h(x,t)+\delta\gamma h(x,t)-\frac{\delta\sqrt{2}e^{-\gamma^{2}t/2}}{\sqrt{\pi t}}\delta_0(x),
\end{equation*}	
with the initial condition $h(x,0)=\delta_0(x)$. Using  the above result in (\ref{gfdt1}), we get
\begin{align*}
\delta\left(\gamma^{2}+2\frac{\mathrm{d}}{\mathrm{d} t}\right)^{1/2}\bar{p}^{f_{3}}(n,t)&=\int_0^{\infty} p^{k}(n,x)\left(-\frac{\partial}{\partial x}h(x,t)+\delta\gamma h(x,t)-\frac{\delta\sqrt{2}e^{-\gamma^{2}t/2}}{\sqrt{\pi t}}\delta_0(x)\right)\mathrm{d}x\\
&=p^{k}(n,0)h(0,t)+\int_0^{\infty}h(x,t)\frac{\mathrm{d}}{\mathrm{d}x}p^{k}(n,x)\mathrm{d}x+\delta\gamma\bar{p}^{f_{3}}(n,t)\\
&\ \  -\frac{\delta\sqrt{2}e^{-\gamma^{2}t/2}}{\sqrt{\pi t}}p^{k}(n,0)\\
&=p^{k}(n,0)h(0,t)+\int_0^{\infty}\bigg(-k(\lambda_{1}+\lambda_{2})p^{k}(n,x)+\lambda_{1}\sum_{j=1}^{k}p^{k}(n-j,x)\\
&\ \ +\lambda_{2}\sum_{j=1}^{k}p^{k}(n+j,x)\bigg)h(x,t)\mathrm{d}x+\delta\gamma\bar{p}^{f_{3}}(n,t)-\frac{\delta\sqrt{2}e^{-\gamma^{2}t/2}}{\sqrt{\pi t}}p^{k}(n,0)\\
&=p^{k}(n,0)h(0,t)-k(\lambda_{1}+\lambda_{2})\bar{p}^{f_{3}}(n,t)+\lambda_{1}\sum_{j=1}^{k}\bar{p}^{f_{3}}(n-j,t)\\
&\ \ +\lambda_{2}\sum_{j=1}^{k}\bar{p}^{f_{3}}(n+j,t)+\delta\gamma\bar{p}^{f_{3}}(n,t)-\frac{\delta\sqrt{2}e^{-\gamma^{2}t/2}}{\sqrt{\pi t}}p^{k}(n,0)\\
&=\left(\delta\gamma-k(\lambda_{1}+\lambda_{2})\right)\bar{p}^{f_{3}}(n,t)+\lambda_{1}\sum_{j=1}^{k}\bar{p}^{f_{3}}(n-j,t)\\
&\ \ +\lambda_{2}\sum_{j=1}^{k}\bar{p}^{f_{3}}(n+j,t) -\delta\gamma \mathrm{Erf}\left(\frac{\gamma\sqrt{t}}{\sqrt{2}}\right)p^{k}(n,0).
\end{align*}	
The last step follows by using the following result (see Vellaisamy and Kumar (2018), Proposition 2.2):
\begin{equation*}
\lim\limits_{x\to0}h(x,t)=h(0,t)=\delta e^{-\gamma^{2}t/2}\left(\sqrt{\frac{2}{\pi t}}-\gamma  e^{\gamma^{2}t/2} \mathrm{Erf}\left(\frac{\gamma\sqrt{t}}{\sqrt{2}}\right)\right).
\end{equation*}
This completes the proof.
\end{proof}
\section{Concluding remarks}
 In this paper, we introduced and studied a fractional version of the SPoK, namely, the FSPoK. It is obtained by time-changing the SPoK with  an independent inverse stable subordinator. Its one-dimensional distribution, pgf, mean, variance and covariance are obtained. It is shown that the FSPoK exhibits the LRD property. Also, we considered two time-changed versions of the FSPoK where time-change is done using a L\'evy subordinator and its inverse. Some distributional properties and particular cases are discussed for these time-changed processes. Kerss {\it et al.} (2014) introduced and studied a fractional version of the Skellam process, namely, the FSP. They showed its applications to high frequency financial data set. As the FSPoK is a generalized version of the FSP, we expect the FSPoK to have potential applications in finance and related fields.

\end{document}